\documentclass[11pt,reqno]{amsart}

\usepackage[l2tabu,orthodox]{nag}

\usepackage[all]{xy}

\usepackage{hyperref}

\usepackage{marginnote}
\usepackage{enumerate}

\usepackage[swedish,english]{babel}
\usepackage[T1]{fontenc}
\usepackage[latin1]{inputenc}

\usepackage{amsaddr}

\usepackage{amsmath,amssymb,amsthm}

\theoremstyle{plain}
\newtheorem{theorem}{Theorem}[section]
\newtheorem{lemma}[theorem]{Lemma}
\newtheorem{prop}[theorem]{Proposition}

\theoremstyle{definition}
\newtheorem{defn}[theorem]{Definition}

\theoremstyle{remark}
\newtheorem{remark}[theorem]{Remark}

\newcommand{\N}{\mathbb{N}}
\newcommand{\Z}{\mathbb{Z}}

\renewcommand{\emptyset}{\varnothing}

\title[Strongly graded Leavitt path algebras]{Strongly graded Leavitt path algebras}
\date{\today}

\begin{document}

\author{Patrik Lundstr\"{o}m}
\address{Department of Engineering Science,\\
University West, SE-46186 Trollh\"{a}ttan, Sweden}
\email{patrik.lundstrom@hv.se}

\author{Johan \"{O}inert}
\address{Department of Mathematics and Natural Sciences,\\
Blekinge Institute of Technology,
SE-37179 Karlskrona, Sweden}
\email{johan.oinert@bth.se}

\subjclass[2010]{16S99, 16W50}
\keywords{strongly graded ring, Leavitt path algebra}

\begin{abstract}
Let $R$ be a unital ring, let $E$ be a directed graph
and recall that the Leavitt path algebra $L_R(E)$
carries a natural $\Z$-gradation.
We show that $L_R(E)$ is strongly $\Z$-graded
if and only if
$E$ is row-finite, has no sink, and satisfies Condition (Y).
Our result generalizes a recent result by Clark, Hazrat and Rigby,
and the proof is short and self-contained.
\end{abstract}

\maketitle


\section{Introduction}

Given an associative unital ring $R$ and a directed graph $E$, one may define
the Leavitt path algebra $L_R(E)$ (see Section~\ref{Subsec:LPA}).
Since their introduction in 2005, Leavitt path algebras have grown into
a common theme within modern 
algebra
on the interface between ring theory and operator algebra.
For an excellent account of the history of the subject and a review of some of its main developments, we refer the reader to Abrams' survey article \cite{AbramsDecade}.

Every Leavitt path algebra comes equipped with a natural $\Z$-gradation (see Remark~\ref{rem:Zgradation})
and there are multiple examples of when the graded structure
has been utilized in the study of $L_R(E)$.
On the other hand, there are only a few examples of when properties of the $\Z$-gradation itself
have been studied.
Notably,
Hazrat has shown that for a finite directed graph $E$,
the Leavitt path algebra $L_R(E)$ with coefficients in a unital ring $R$
is strongly $\Z$-graded if and only if $E$ has no sink (see \cite[Theorem 3.15]{H2013A}).
The authors of the present article have shown
that if $E$ is a finite directed graph, then $L_R(E)$ is always \emph{epsilon-strongly $\Z$-graded} (see \cite[Theorem 1.2]{NO2020}).

In their recent article \cite{RigbyEtAl},
Clark, Hazrat and Rigby introduced the following condition.

\begin{defn}[Clark et al \cite{RigbyEtAl}]\label{def:CondY}
A directed graph $E$ satisfies
\emph{Condition (Y)} if for every $k \in \N$
and every infinite path $p$, there exists
an initial subpath $\alpha$ of $p$ and a finite path
$\beta$ such that $r(\beta)=r(\alpha)$ and $|\beta|-|\alpha|=k$.
\end{defn}

Using results on Steinberg algebras, 
in the case when $K$ is a commutative unital ring,
they showed that $L_K(E)$ is strongly $\Z$-graded
if and only if
$E$ is row-finite, has no sink, and satisfies Condition (Y)
(see \cite[Theorem 4.2]{RigbyEtAl}).

We now introduce the following seemingly weaker condition,
which can actually be shown to be equivalent to Condition (Y).

\begin{defn}\label{def:CondY1}
A directed graph $E$ satisfies
\emph{Condition (Y1)} if for every infinite path $p$, there exists
an initial subpath $\alpha$ of $p$ and a finite path
$\beta$ such that $r(\beta)=r(\alpha)$ and $|\beta|-|\alpha|=1$.
\end{defn}

Our main result is the following generalization of
\cite[Theorem 4.2]{RigbyEtAl}.

\begin{theorem}\label{thm:main}
Let $R$ be a unital, but not necessarily commutative, ring and let $E$ be a directed graph.
Consider the Leavitt path algebra $L_R(E)$ with its canonical $\Z$-gradation.
The following three assertions are equivalent:
\begin{enumerate}[{\rm (i)}]
	\item $L_R(E)$ is strongly $\Z$-graded;
	\item $E$ is row-finite, has no sink, and satisfies Condition (Y);
	\item $E$ is row-finite, has no sink, and satisfies Condition (Y1).
\end{enumerate}
\end{theorem}

In the context of graph C$^*$-algebras, Chirvasitu \cite[Theorem 2.14]{Chirvasitu}
has shown that the assertions on $E$ made in Theorem~\ref{thm:main}(ii)
are equivalent to freeness of the full gauge action on the
graph C$^*$-algebra $C^*(E)$.

Here is an outline of this article.

In Section~\ref{Sec:Prel} we record definitions and results that will be used in the sequel.
In Section~\ref{Sec:Nec}, we prove Lemma~\ref{lemma:nec}
which establishes the implication (i)$\Rightarrow$(ii) in Theorem~\ref{thm:main}.
Notice that the implication (ii)$\Rightarrow$(iii) is trivial.
In Section~\ref{Sec:Suff}, we prove Proposition~\ref{prop:suff} 
which establishes the implication (iii)$\Rightarrow$(i) in Theorem~\ref{thm:main},
thereby finishing the proof of Theorem~\ref{thm:main}.
In Section~\ref{sec:Examples}, we illustrate Condition (Y)
with numerous concrete examples of directed graphs.

\begin{remark}
Whereas the proof of \cite[Theorem 4.2]{RigbyEtAl} uses a lot of highly technical results 
concerning the so-called path groupoid from \cite{kumjian} as well as the existence of
an isomorphism between Leavitt path algebras and Steinberg algebras from \cite{clark}, 
our approach gives a direct proof that holds in a more general situation.
We are also able to distil down all of the topological machinery from \cite{RigbyEtAl} into a well-known result 
concerning non-emptiness of an inverse limit of finite sets (see Lemma \ref{lemma:handy}).
\end{remark}

\section{Preliminaries}\label{Sec:Prel}

In this section we will record definitions and results that will be used in the rest of this article.

\subsection{Strongly $\Z$-graded rings}

Recall that an associative ring $S$ is said to be \emph{$\Z$-graded}
if, for each $n\in \Z$, there is an additive subgroup $S_n$ of $S$
such that
$S = \oplus_{n\in \Z} S_n$
and for all $n,m\in \Z$ the inclusion $S_n S_m \subseteq S_{n+m}$ holds.
If, in addition, for all $n,m\in \Z$, the equality $S_n S_m = S_{n+m}$ holds,
then $S$ is said to be \emph{strongly $\Z$-graded}.

Throughout the rest of this subsection $S$ denotes a, not necessarily unital, $\mathbb{Z}$-graded ring.

\begin{prop}\label{prop:stronglyZgraded}
The ring $S$ is strongly $\Z$-graded if and only if 
$S_0 S_n = S_n = S_n S_0$, for every $n \in \mathbb{Z}$, 
and the equalities $S_1 S_{-1} = S_{-1} S_1 = S_0$ hold. 
\end{prop}

\begin{proof}
The ''only if'' statement is immediate.
Now we show the ''if'' statement.
Take positive integers $m$ and $n$.
First we show by induction that $S_m = (S_1)^m$.
The base case $m=1$ is clear. Next suppose that 
$S_m = (S_1)^m$. Then we get that 
$S_{m+1} = S_0 S_{m+1} = S_1 S_{-1} S_{m+1} \subseteq
S_1 S_{-1 + m+1} = S_1 S_m = S_1 (S_1)^m = (S_1)^{m+1} \subseteq S_{m+1}$.
Next we show by induction that $S_{-n} = (S_{-1})^n$.
The base case $n=1$ is clear. Next suppose that 
$S_{-n} = (S_{-1})^n$. Then we get that 
$S_{-n-1} = S_{-n-1} S_0 = S_{-n-1} S_1 S_{-1} \subseteq 
S_{-n-1+1} S_{-1} = S_{-n} S_{-1} = (S_{-1})^n S_{-1} = (S_{-1})^{n+1} \subseteq S_{-n-1}$.

Case 1: $S_m S_n = (S_1)^m (S_1)^n = (S_1)^{m+n} = S_{m+n}$.

Case 2: $S_{-m} S_{-n} = (S_{-1})^m (S_{-1})^n = (S_{-1})^{m+n} = S_{-m-n}$.

Case 3: Now we show that $S_m S_{-n} = S_{m-n}$.
We get that 
$S_m S_{-n} = (S_1)^m (S_{-1})^n$.
By repeated application of the equality $S_1 S_{-1} = S_0$,
we get that $(S_1)^m (S_{-1})^n = (S_1)^{m-n} = S_{m-n}$,
if $m \geq n$, or $(S_1)^m (S_{-1})^n = (S_{-1})^{n-m} = S_{m-n}$, otherwise.

Case 4: $S_{-m} S_n = S_{n-m}$. This is shown in a similar fashion
to Case 3, using the equality $S_{-1} S_1 = S_0$, and is therefore left to the reader.
\end{proof}

\subsection{Leavitt path algebras}\label{Subsec:LPA}

Let $R$ be an associative unital ring and let
$E = (E^0,E^1,r,s)$ be a directed graph.
Recall that 
$r$ (range) and $s$ (source) are maps $E^1 \to E^0$. The elements of $E^0$ are called \emph{vertices} 
and the elements of $E^1$ are called \emph{edges}. 
A vertex $v$ for which $s^{-1}(v)$ is empty is called a \emph{sink}.
A vertex $v$ for which $r^{-1}(v)$ is empty is called a \emph{source}.
If $s^{-1}(v)$ is a finite set for every $v \in E^0$,
then $E$ is called \emph{row-finite}.
If $s^{-1}(v)$ is an infinite set, then $v\in E^0$
is called an \emph{infinite emitter}.
If both $E^0$ and $E^1$ are finite sets, then we say that $E$ is \emph{finite}.
A \emph{path} $\mu$ in $E$ is a sequence of edges 
$\mu = \mu_1 \ldots \mu_n$ such that $r(\mu_i)=s(\mu_{i+1})$ 
for $i\in \{1,\ldots,n-1\}$. In such a case, $s(\mu):=s(\mu_1)$ 
is the \emph{source} of $\mu$, $r(\mu):=r(\mu_n)$ is the \emph{range} 
of $\mu$, 
and $|\mu|:=n$ is the \emph{length} of $\mu$.
For any vertex $v \in E^0$ we put $s(v):=v$ and $r(v):=v$.
The elements of $E^1$ are called \emph{real edges}, while for $f\in E^1$
we call $f^*$ a \emph{ghost edge}.
The set $\{f^* \mid f \in E^1\}$ will be denoted by $(E^1)^*$.
We let $r(f^*)$ denote $s(f)$, and we let $s(f^*)$ denote $r(f)$.
For $n \geq 2$, we define $E^n$ to be the set of paths of length $n$, and
$E^* = \cup_{n\geq 0} E^n$ is the set of all finite paths.
If $\mu = \mu_1 \mu_2 \mu_3 \ldots$
where $\mu_i \in E^1$, for all $i\in \N$,
and $r(\mu_i)=s(\mu_{i+1})$ for all $i\in \N$,
then $\mu$ is said to be an \emph{infinite path}.
The set of all infinite paths is denoted by $E^\infty$.
If $p \in E^* \cup E^\infty$ and some $\alpha \in E^*$, $p' \in E^* \cup E^\infty$
satisfy $p=\alpha p'$,
then $\alpha$ is said to be an \emph{initial subpath of $p$}.

Following Hazrat \cite{H2013A} we make the following definition.

\begin{defn}\label{def:LPA}
The \emph{Leavitt path algebra of $E$ with coefficients in $R$}, denoted by $L_R(E)$,
is the algebra generated by the sets
$\{v \mid v\in E^0\}$, $\{f \mid f\in E^1\}$ and $\{f^* \mid f\in E^1\}$
with the coefficients in $R$,
subject to the relations:
\begin{enumerate}
	\item $uv = \delta_{u,v} v$ for all $u,v \in E^0$;
	\item $s(f)f=fr(f)=f$ and $r(f)f^*=f^*s(f)=f^*$, for all $f\in E^1$;
	\item $f^*f'=\delta_{f,f'} r(f)$, for all $f,f'\in E^1$;
	\item $\sum_{f\in E^1, s(f)=v} ff^* = v$, for every $v\in E^0$ for which $s^{-1}(v)$ is non-empty and finite.
\end{enumerate}
Here the ring $R$ commutes with the generators.
\end{defn}

\begin{remark}\label{rem:Zgradation}
The Leavitt path algebra $L_R(E)$ carries a natural $\Z$-gradation.
Indeed, put $\deg(v)=0$ for each $v\in E^0$.
For each $f\in E^1$ we put $\deg(f)=1$ and $\deg(f^*)=-1$.
By assigning degrees to the generators in this way,
we obtain a $\Z$-gradation on the free algebra $F_R(E) = R \langle	v, f, f^* \mid v \in E^0, f \in E^1 \rangle$.
Moreover, the ideal coming from relations (1)--(4) in Definition~\ref{def:LPA} is homogeneous. Using this it is easy to see that the natural $\Z$-gradation on $F_R(E)$ carries over to a $\Z$-gradation on the quotient algebra $L_R(E)$.
\end{remark}

\subsection{Inverse limits of finite sets}

The following lemma follows from a general result concerning non-empti\-ness of inverse limits
in \cite[Chapter III § 7.4]{bourbaki}. For the convenience of the reader, we give a short direct proof
adapted to
the situation at hand.

\begin{lemma}\label{lemma:handy}
If $(X_n)_{n \in \mathbb{N}}$ is a sequence of finite non-empty sets and for all $n \in \mathbb{N}$,
$g_n$ is a function $X_{n+1} \to X_n$,
then there exists an element $(x_1,x_2,x_3,\ldots) \in \prod_{n \in \mathbb{N}} X_n$ such that for each $n \in \mathbb{N}$ the equality $g_n(x_{n+1}) = x_n$ holds.
\end{lemma}

\begin{proof}
We claim that there exists a sequence of sets $(Z_n)_{n \in \mathbb{N}}$ such that
for all $n \in \mathbb{N}$, $Z_n$ is a non-empty subset of $X_n$ and $g_n( Z_{n+1} ) = Z_n$. 
Let us assume for a moment that the claim holds.
Define an element $(x_1, x_2, x_3, \ldots) \in \prod_{n \in \mathbb{N}} Z_n$ inductively in the following way.
Let $x_1$ be any element in $Z_1$. Take $m \in \mathbb{N}$. Suppose that we have defined $x_n \in Z_n$
for all $n \leq m$. Then let $x_{m+1}$ be any element in $g_m^{-1}( x_m ) \cap Z_{m+1}$.
It is clear that the element $(x_1, x_2, x_3,\ldots)$ has the desired properties.
Now we show the claim. For all $m,n \in \mathbb{N}$ put 
\begin{displaymath}
		Y_n^m = ( g_n \circ g_{n+1} \circ \cdots \circ g_{m+n-1} ) ( X_{m+n} ).
\end{displaymath}
Then for all $m,n \in \mathbb{N}$, the set $Y_n^m$ is a finite and non-empty subset of $X_n$, and the relations
\begin{equation}\label{inclusion}
Y_n^{m+1} \subseteq Y_n^m
\end{equation}
and
\begin{equation}\label{relation}
g_n( Y_{n+1}^m ) = Y_n^{m+1}
\end{equation}
hold. For all $n \in \mathbb{N}$ put $Z_n = \bigcap_{m \in \mathbb{N}} Y_n^m$.
From \eqref{inclusion} it follows that every $Z_n$ is a finite and non-empty subset of $X_n$.
In fact, for all $n \in \mathbb{N}$, there is $p(n) \in \mathbb{N}$ with the property that
for all $k \geq p(n)$, the equalities $Y_n^k = Y_n^{p(n)} = Z_n$ hold.
Take $n \in \mathbb{N}$ and $k = {\rm max}( p(n+1) , p(n) )$. 
Then, from \eqref{relation}, we get that
\begin{displaymath}
		g_n( Z_{n+1} ) = g_n( Y_{n+1}^{p(n+1)} ) = g_n( Y_{n+1}^k ) = Y_n^{k+1} = Z_n
\end{displaymath}
which shows the claim.
\end{proof}

\section{Necessary conditions}\label{Sec:Nec}

In this section we will prove Lemma~\ref{lemma:nec}
which establishes the implication (i)$\Rightarrow$(ii) in Theorem~\ref{thm:main}.

\begin{lemma}\label{lemma:nec}
Let $R$ be a unital ring and let $E$ be a directed graph.
Consider the Leavitt path algebra $L_R(E)$ with its canonical $\Z$-gradation.
If $L_R(E)$ is strongly $\Z$-graded, then the following three assertions hold:
\begin{enumerate}[{\rm (i)}]
	\item $E$ has no sink;
	\item $E$ is row-finite;
	\item $E$ satisfies Condition (Y).
\end{enumerate}
\end{lemma}

\begin{proof}
Suppose that $S=L_R(E)$ is strongly $\Z$-graded.

(i)
Seeking a contradiction, suppose that there is a sink $v$ in $E$.
Then $v \in S_0 = S_1 S_{-1}$.
Using that $v$ is a sink, we get that $v = v^2 \in v S_1 S_{-1} = \{0\}$.
This is a contradiction.

(ii)
Let $v \in E^0$ be an arbitrary vertex.
From the strong gradation we get that $v \in S_0 = S_1 S_{-1}$, i.e. $v = \sum_{i=1}^n \alpha_i \beta_i^* \gamma_i \delta_i^*$ where $\alpha_i \beta_i^* \in S_1$ and $\gamma_i \delta_i^* \in S_{-1}$.
Notice that $|\delta_i|>0$ for each $i$.
Seeking a contradiction, suppose that $v$ is an infinite emitter.
Then 
$f = vf = \sum_{i=1}^n \alpha_i \beta_i^* \gamma_i \delta_i^* f$
for infinitely many $f$'s.
But that is not possible since $n < \infty$.
This is a contradiction.
We conclude that $E$ is row-finite.

(iii)
Let $p$ be an infinite path and let $k>0$ be an arbitrary integer.
Put $v=s(p)$.
Using that $S_0 = S_{-k}S_{k}$ we may write
\begin{displaymath}
	v = \sum_{i=1}^{n} \alpha_i \beta_i^* \gamma_i \delta_i^*
\end{displaymath}
where $\alpha_i \beta_i^* \in S_{-k}$ and $\gamma_i \delta_i^* \in S_{k}$.
Let $p'$ be an initial subpath of $p$ such that $|p'|>|\delta_i|$ for each $i$.
Clearly, $vp'=s(p')p'=p'$. Hence there must be some $m$ such that $\delta_m$ is an initial subpath of $p'$ (and thus also of $p$), for otherwise we would have ended up with $vp'=0$.
Using the notation of Definition~\ref{def:CondY}),
put $\alpha := \delta_m$ and $\beta := \gamma_m$
and notice that
$r(\delta_m)=r(\gamma_m)$ and $|\gamma_m|-|\delta_m|=k$.
This shows that $E$ satisfies Condition (Y).
\end{proof}

\section{Sufficient conditions}\label{Sec:Suff}

In this section we will prove Proposition~\ref{prop:suff} 
which establishes the implication (iii)$\Rightarrow$(i) in Theorem~\ref{thm:main}.

\begin{lemma}\label{lemma:suff1}
Let $R$ be a unital ring and let $E$ be a directed graph.
Consider the Leavitt path algebra $S=L_R(E)$ with its canonical $\Z$-gradation.
If $E$ is row-finite and has no sink, then $S_1 S_{-1} = S_0$.
\end{lemma}

\begin{proof}
Clearly, $S_1 S_{-1} \subseteq S_0$.
Notice that $S_1 S_{-1}$ is an ideal of $S_0$,
and that certain sums of elements of $E^0$ 
form a set of local units for $L_R(E)$, and in particular for $S_0$.
Thus, if we can show that $E^0 \subseteq S_1 S_{-1}$,
then it follows that $S_0 \subseteq S_1 S_{-1}$.
Now take $v\in E^0$.
Then $v = \sum_{s(f)=v} ff^* \in S_1 S_{-1}$.
\end{proof}

\begin{defn}
If $\alpha \in E^n$, for some $n\in \N$, then we say that 
$r(\alpha)$ is a \emph{turning node} for $\alpha$ if there exists
$\beta \in E^{n+1}$ with $r(\alpha) = r(\beta)$.
\end{defn}

\begin{prop}\label{prop:suff}
Let $R$ be a unital ring and let $E$ be a directed graph.
Consider the Leavitt path algebra $L_R(E)$ with its canonical $\Z$-gradation.
If $E$ is row-finite, has no sink, and satisfies Condition (Y1),
then $L_R(E)$ is strongly $\Z$-graded. 
\end{prop}

\begin{proof}
Put $S=L_R(E)$.
Notice that, by Lemma~\ref{lemma:suff1}, $S_0 = S_1 S_{-1}$.
Using that $S$ has a set of local units which is contained in $S_0$,
it is clear that $S_0 S_n = S_n = S_n S_0$ for every $n\in \Z$.
In view of Proposition~\ref{prop:stronglyZgraded} it remains to show that $S_0 = S_{-1} S_1$.
By the same argument as in the proof of Lemma~\ref{lemma:suff1},
it is sufficient to show that $E^0 \subseteq S_{-1} S_1$.
Clearly, any vertex $v\in E^0$ which is not a source belongs to $S_{-1} S_1$, since $v=r(f)=f^*f$ for some $f\in E^1$.
Thus, in order to show that $L_R(E)$ is strongly $\Z$-graded it remains to show that $v\in S_{-1} S_1$ for every source $v$.

Suppose that $v$ is a source. Using $v$ we will now inductively define
a sequence $(X_n)_{n \in \mathbb{N}}$ of sets in the following way.
Put
\begin{displaymath}
	X_1 = \{ f \in E^1 \mid s(f)=v, \mbox{and $r(f)$ is not a turning node for $f$} \}
\end{displaymath}
and notice that it is a finite set since $E$ is row-finite.
Suppose that we have defined the finite set $X_n \subseteq E^n$ for some $n \in \mathbb{N}$.
Put 
\begin{align*}
	X_{n+1} = \{ \alpha f \in E^{n+1} \mid & \alpha \in X_n, f \in E^1, s(f)=r(\alpha), \\
	& \text{ and } r(f) \text{ is not a turning node for } \alpha f \}.
\end{align*}
By finiteness of $X_n$ and row-finiteness of $E$ we conclude that $X_{n+1}$ is finite.
Seeking a contradiction, suppose that $X_n$ is non-empty for every $n \in \mathbb{N}$.
For each $n \in \mathbb{N}$, we define a function $g_n : X_{n+1} \to X_n$ by putting
$g_n(\alpha f)=\alpha$, for $\alpha f \in X_{n+1}$.
By Lemma~\ref{lemma:handy}, 
there exists an element $(x_1,x_2,x_3,\ldots) \in \prod_{n \in \mathbb{N}} X_n$ such that for all $n \in \mathbb{N}$ the equality $g_n(x_{n+1}) = x_n$ holds.
In other words, $x_1 x_2 x_3 \cdots$ is an infinite path in $E$
such that $s(x_1)=v$ and with the property that
for all $n \in \mathbb{N}$, $r(x_n)$ is not a turning node
for $x_1 x_2 \cdots x_n$. This contradicts Condition (Y1).
Therefore, for some $n \in \mathbb{N}$, the set $X_n$ is empty. 
Put $k = \min\{n \in \N \mid X_n =\emptyset \}$.

We claim that
$v=\sum_{i=1}^m \alpha_i \alpha_i^*$
for some $m\in \N$ and $\alpha_1,\ldots,\alpha_m \in E^*$
such that, for each $i\in \{1,\ldots,m\}$, $r(\alpha_i)$ is a turning node for $\alpha_i$.
If we assume that the claim holds,
then for each $i\in \{1,\ldots,m\}$ there is some $\beta_i \in E^*$
such that $r(\alpha_i)=r(\beta_i)$ and $|\beta_i|-|\alpha_i|=1$.
Thus,
\begin{displaymath}
	v=\sum_{i=1}^m \alpha_i \alpha_i^*
	= \sum_{i=1}^m \alpha_i r(\alpha_i) \alpha_i^*
	= \sum_{i=1}^m \alpha_i \beta_i^* \beta_i \alpha_i^* \in S_{-1} S_1
\end{displaymath}
as desired.

Now we show the claim.
Using that $E$ is row-finite and that $v$ is not a sink, we may write
\begin{equation}
	v = \sum_{i=1}^{m'} f_i f_i^*
\label{eq:vSum}
\end{equation}
with $\{f_1,\ldots,f_{m'}\} = s^{-1}(v)$.
If, for some $i$, $r(f_i)$ is not a turning node for $f_i$,
then we may replace $f_i f_i^* = f_i r(f_i) f_i^* = f_i \sum_{h\in s^{-1}(r(f_i))} hh^* f_i^*$
in Equation~\eqref{eq:vSum}.
By repeating this procedure (if necessary) it is clear that
we, in a finite number of steps, will be able to identify $m\in \N$ and $\alpha_1,\ldots,\alpha_m \in E^*$
such that $|\alpha_i|\leq k$,
and with the properties that
$v=\sum_{i=1}^m \alpha_i \alpha_i^*$
and for each $i\in \{1,\ldots,m\}$, $r(\alpha_i)$ is a turning node for $\alpha_i$.
\end{proof}

\section{Examples}\label{sec:Examples}

In this section, we illustrate Condition (Y) with some concrete examples of (directed) graphs.
If the graph does not contain any infinite paths, for instance if it is finite, then it trivially satisfies Condition (Y).
Therefore, we will, from now on, only consider examples of graphs containing infinite paths.

\subsection{Examples of directed graphs which satisfy Condition (Y)}

As noted in \cite[Section 4.1.2]{RigbyEtAl}, if $E$ is a graph such that every infinite path contains a 
vertex that is the base of a cycle, then $E$ satisfies Condition (Y).
Here is an example of such a graph: 
\\
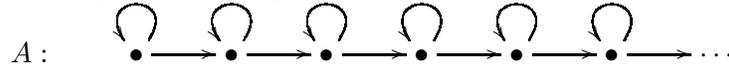
\begin{displaymath}
	\xymatrix{
A: & \bullet\ar@(ur,ul)\ar[r]  & \bullet\ar@(ur,ul)\ar[r] & \bullet\ar@(ur,ul)\ar[r] & \bullet\ar@(ur,ul)\ar[r]
& \bullet\ar@(ur,ul)\ar[r] & \bullet\ar@(ur,ul)\ar[r] & \cdots \\
	}	
\end{displaymath}
It is clear that we can erase as many loops as we wish in the graph $A$ and still get a graph satisfying Condition (Y),
as long as we keep infinitely many loops. For instance, the graph
\\
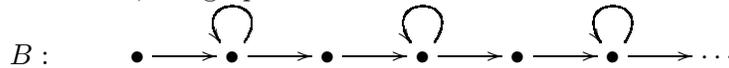
\begin{displaymath}
	\xymatrix{
B: & \bullet\ar[r]  & \bullet\ar@(ur,ul)\ar[r] & \bullet\ar[r] & \bullet\ar@(ur,ul)\ar[r]
& \bullet\ar[r] & \bullet\ar@(ur,ul)\ar[r] & \cdots \\
	}	
\end{displaymath}
satisfies Condition (Y). There are graphs without cycles which still satisfy Condition (Y).
The most ``generic'' example is probably the following graph:
\\
\begin{displaymath}
	\xymatrix{
C: & \cdots\ar[r]  & \bullet\ar[r] & \bullet\ar[r] & \bullet\ar[r] & \bullet\ar[r] & \bullet\ar[r] & \cdots \\
	}	
\end{displaymath}
We can also construct ``left finite'' versions of the graph $C$ which still satisfy Condition (Y)
if we add sufficiently long sequences of vertical arrows at each step, for instance:
\\
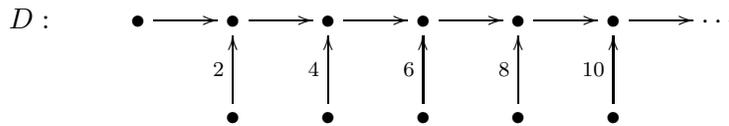
\begin{displaymath}
	\xymatrix{
D: & \bullet\ar[r]  & \bullet\ar[r]   & \bullet\ar[r]   & \bullet\ar[r]   & \bullet\ar[r]   & \bullet\ar[r]      & \cdots  \\
   &                & \bullet\ar[u]^2 & \bullet\ar[u]^4 & \bullet\ar[u]^6 & \bullet\ar[u]^8 & \bullet\ar[u]^{10} & \\ 
	}	
\end{displaymath}
To ease the notation we let the above numbers $2, 4, 6, \ldots$ indicate that there are $2, 4, 6,\ldots$ vertical arrows in the respective indicated column, forming line paths of length $2, 4, 6$ and so on.
We can also make the graph $D$ ``thinner'' and still make it satisfy Condition (Y), for instance:
 \\
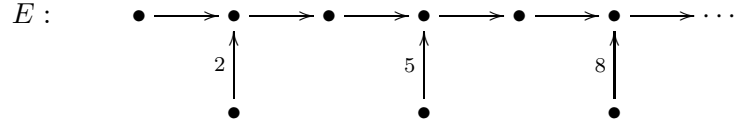
\begin{displaymath}
	\xymatrix{
E: & \bullet\ar[r]  & \bullet\ar[r]   & \bullet\ar[r] & \bullet\ar[r]   & \bullet\ar[r] & \bullet\ar[r]      & \cdots  \\
   &                & \bullet\ar[u]^2 &               & \bullet\ar[u]^5 &               & \bullet\ar[u]^8 & \\ 
	}	
\end{displaymath}
We can construct an ``odd'' version of the graph $D$ which satisfies Condition (Y):
\\
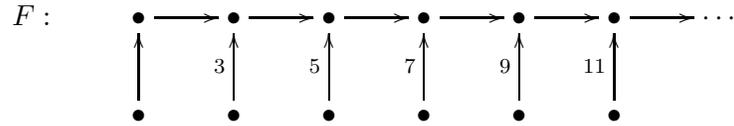
\begin{displaymath}
	\xymatrix{
F: & \bullet\ar[r]  & \bullet\ar[r]   & \bullet\ar[r]   & \bullet\ar[r]   & \bullet\ar[r]   & \bullet\ar[r]      & \cdots  \\
   & \bullet\ar[u]  & \bullet\ar[u]^3 & \bullet\ar[u]^5 & \bullet\ar[u]^7 & \bullet\ar[u]^9 & \bullet\ar[u]^{11} & \\ 
	}	
\end{displaymath}
We can also make ``thinner'' versions of the graph $F$ which still satisfy Condition (Y), for instance:
\\
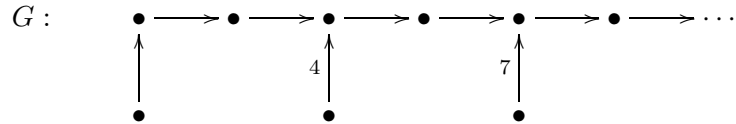
\begin{displaymath}
	\xymatrix{
G: & \bullet\ar[r]  & \bullet\ar[r]   & \bullet\ar[r]   & \bullet\ar[r]   & \bullet\ar[r]   & \bullet\ar[r]      & \cdots  \\
   & \bullet\ar[u]  &                 & \bullet\ar[u]^4 &                 & \bullet\ar[u]^7 &  & \\ 
	}	
\end{displaymath}

\subsection{Examples of directed graphs which do not satisfy Condition (Y)}
The most ``generic'' example of a graph which does not satisfy Condition (Y) is probably:
\\
\begin{displaymath}
	\xymatrix{
H: & \bullet\ar[r] & \bullet\ar[r] & \bullet\ar[r] & \bullet\ar[r] & \bullet\ar[r] & \bullet\ar[r] & \cdots \\
	}	
\end{displaymath}
If we remove sufficiently many arrows from all but finitely many of the columns in any of the examples
$D$, $E$, $F$ and $G$, then we get graphs which do not satisfy Condition (Y), for instance:
\\
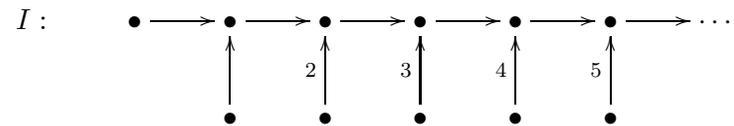
\begin{displaymath}
	\xymatrix{
I: & \bullet\ar[r]  & \bullet\ar[r]   & \bullet\ar[r]   & \bullet\ar[r]   & \bullet\ar[r]   & \bullet\ar[r]   & \cdots  \\
   &                & \bullet\ar[u]   & \bullet\ar[u]^2 & \bullet\ar[u]^3 & \bullet\ar[u]^4 & \bullet\ar[u]^5 & \\ 
	}	
\end{displaymath}
 \\
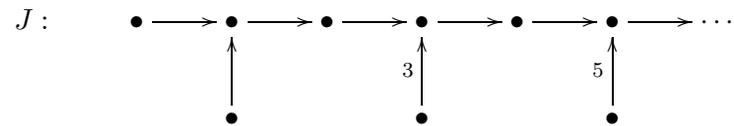
\begin{displaymath}
	\xymatrix{
J: & \bullet\ar[r]  & \bullet\ar[r] & \bullet\ar[r] & \bullet\ar[r]   & \bullet\ar[r] & \bullet\ar[r]      & \cdots  \\
   &                & \bullet\ar[u] &               & \bullet\ar[u]^3 &               & \bullet\ar[u]^5 & \\ 
	}	
\end{displaymath}
\\
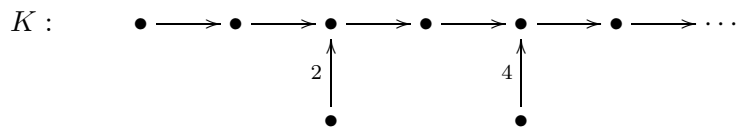
\begin{displaymath}
	\xymatrix{
K: & \bullet\ar[r]  & \bullet\ar[r]   & \bullet\ar[r]   & \bullet\ar[r]   & \bullet\ar[r]   & \bullet\ar[r] & \cdots  \\
   &                &                 & \bullet\ar[u]^2 &                 & \bullet\ar[u]^4 &               & \\ 
	}	
\end{displaymath}


\begin{thebibliography}{99}


\bibitem{AbramsDecade}
G. Abrams,
Leavitt path algebras: the first decade,
{\it Bull. Math. Sci.} \textbf{5} (2015), no. 1, 59--120.

\bibitem{bourbaki}
N. Bourbaki,
{\it Elements of mathematics. Theory of sets.},
Translated from the French Hermann, Publishers in Arts and Science, Paris; Addison-Wesley Publishing Co., Reading, Mass.-London-Don Mills, Ont. (1968), viii+414 pp.

\bibitem{Chirvasitu}
A. Chirvasitu,
Gauge freeness for Cuntz-Pimsner algebras,
arXiv:1805.12318 [math.OA]

\bibitem{RigbyEtAl}
L. O. Clark, R. Hazrat, S. W. Rigby,
Strongly graded groupoids and strongly graded Steinberg algebras,
{\it J. Algebra} \textbf{530} (2019), 34--68.

\bibitem{clark}
L.O. Clark and A. Sims, 
Equivalent groupoids have Morita equivalent Steinberg algebras, 
{\it J. Pure Appl. Algebra} \textbf{219} (2015), 2062--2075.

\bibitem{H2013A}
R. Hazrat,
The graded structure of Leavitt path algebras,
{\it Israel J. Math.} \textbf{195} (2013), no. 2, 833--895.

\bibitem{kumjian}
A. Kumjian, D. Pask, I. Raeburn and J. Renault, 
Graphs, groupoids, and Cuntz-Krieger algebras, 
{\it J. Funct. Anal.} \textbf{144} (1997), 505--541.

\bibitem{NO2020}
P. Nystedt, J. \"{O}inert,
Group gradations on Leavitt path algebras,
{\it J. Algebra Appl.} \textbf{19} (2020), no. 9, 2050165.


\end{thebibliography}
\end{document}